\documentclass[12pt]{amsart}
\usepackage[parfill]{parskip}    
\usepackage{amsfonts}
\usepackage{amsmath, amssymb}
\usepackage{amsthm}
\usepackage{epstopdf}
\usepackage{rlepsf, hyperref}
\usepackage[all]{xy}
\usepackage{psfrag}
\usepackage{times}

\title[Uniqueness of area minimizing surfaces]{Uniqueness of area minimizing surfaces for extreme curves}
\author{Baris Coskunuzer}
\author{Tolga Etg\"u}
\address{Department of Mathematics, Ko\c{c} University, Istanbul 34450 TURKEY }
\email{bcoskunuzer@ku.edu.tr}
\email{tetgu@ku.edu.tr}
\thanks{The first author is partially supported by EU-FP7 Grant IRG-226062, TUBITAK Grant 109T685 and TUBA-GEBIP Award. The second author is partially supported by TUBITAK Grant 110T209.}

\newtheorem*{thm*}{Theorem} \theoremstyle{definition}

\newcommand{\Z}{\mathbb{Z}}

\newcommand{\N}{\mathcal{E}}
\newcommand{\BR}{\mathbb{R}}
\newcommand{\C}{\mathcal{G}}

\newcommand{\U}{\mathcal{U}}
\newcommand{\B}{\mathcal{F}}
\newcommand{\T}{\mathcal{V}}
\newcommand{\wh}{\widehat}

\begin{document}

\newtheorem{thm}{Theorem}[section]
\newtheorem{lem}[thm]{Lemma}
\newtheorem{cor}[thm]{Corollary}
\newtheorem{prop}[thm]{Proposition}

\theoremstyle{remark}
\newtheorem{rmk}[thm]{Remark}
\newtheorem{defn}[thm]{\bf{Definition}}

\begin{abstract}
Let $M$ be a compact, orientable, mean convex $3$-manifold with boundary $\partial M$. We show that the set of all simple closed curves  in $\partial
M$ which bound unique area minimizing disks in $M$ is dense in the space of  simple closed curves  in $\partial M$ which are nullhomotopic in $M$. We
also show that the set of all simple closed curves in $\partial M$ which bound unique absolutely area minimizing surfaces in $M$ is dense in the
space of simple closed curves in $\partial M$ which are  nullhomologous in $M$.

\end{abstract}

\maketitle

\section{Introduction}

The Plateau problem investigates the existence of an area minimizing disk (or surface) for a given curve in a given manifold $M$. Besides the solution of this problem, there have been many
important results on the regularity, embeddedness and the number of solutions as well. In this paper, we focus on the number of solutions and give new uniqueness results.

The main question along this line is if, for a given curve, there is a unique
area minimizing disk or surface in the ambient manifold $M$. The first result about this question came from Rado in early 1930s. He showed that if a
curve can be projected bijectively to a convex plane curve, then it bounds a unique minimal disk \cite{Ra}. Then in the early 1970s, Nitsche proved
uniqueness of minimal disks for boundary curves with total curvature less than $4\pi$ in \cite{Ni}. Then, Tromba \cite{Tr} showed that a generic
curve in $\BR^3$ bounds a unique area minimizing disk. Morgan \cite{Mo} proved a similar result concerning absolutely area minimizing surfaces.
Later, White proved a very strong generic uniqueness result for fixed topological type in any dimension and codimension \cite{Wh1}. In particular, he
showed that a generic $k$-dimensional $C^{j,\alpha}$-submanifold of a Riemannian manifold cannot bound two smooth, minimal $(k+1)$-manifolds of equal
area.

In \cite{Co1}, the first author proved generic uniqueness results for both versions of the Plateau problem under the condition that $H_2(M;\Z)=0$. In
this paper, we generalize these results by removing the assumption on homology. Our techniques are simple and topological. The first main result is
the following:

\vspace{.1in}

\noindent {\bf Theorem 3.1.} Suppose that $M$ is a compact, orientable, mean convex $3$-manifold. Let $\N$ be the set of simple closed curves on the
boundary of $M$ which are nullhomotopic in $M$, and $\U \subset \N$ consist of those which bound unique area minimizing disks in $M$. Then $\U$ is
not only dense but also a countable intersection of open dense subsets of $\N$ with respect to the $C^0$-topology.

\vspace{.1in}

The second main result is a similar theorem for absolutely area minimizing surfaces:

\vspace{.1in}

\noindent {\bf Theorem 4.1.} Suppose that $M$ is a compact, orientable, mean convex $3$-manifold. Let $\B$ be the set of simple closed curves on the
boundary of $M$ which are nullhomologous in $M$, and $\T \subset \B$ consist of those which bound unique absolutely area minimizing surfaces in $M$.
Then, $\T$ is not only dense but also countable intersection of open dense subsets of $\B$ with respect to the $C^0$-topology.

\vspace{.1in}

For natural generalizations of these results to the smooth category see the last section of this paper.

The ``lens" technique introduced in \cite{Co1} to prove generic uniqueness results does not generalize to manifolds with nontrivial homology, mainly
because the disks or surfaces may not be separating in $M$ in general, hence one cannot construct a canonical neighborhood ({\em lens})
$N_\Gamma=[\Sigma^-_\Gamma,\Sigma^+_\Gamma]$ which contains all area minimizing disks (or surfaces) for a given nullhomotopic (or nullhomologous)
$\Gamma\subset \partial M$. Provided that these neighborhoods exist and are disjoint for disjoint curves on the boundary, a summation argument which
involves the thickness (or volume) of $N_\Gamma$ would give the desired uniqueness results. But these lenses are the key element in the proof, and
without them, the whole argument collapses.

In the disk case, we still have the disjointness of the area minimizing disks for disjoint boundaries in general by \cite{MY2}. Even though we could
not construct disjoint lenses $N_\Gamma$ for a given curve $\Gamma\subset \partial M$ as in \cite{Co1} because of nontrivial homology, when we
consider the behavior of area minimizing disks near the boundary, we still get disjoint canonical neighborhoods ({\em lens with a big hole})  near
$\partial M$ for disjoint curves, and the summation argument in \cite{Co1} works. Hence the proof of Theorem \ref{disk} can be achieved with a
modification of the original argument in \cite{Co1}.

On the other hand, in the surface case, we do not have the disjointness of the absolutely area minimizing surfaces for disjoint boundaries when the
ambient manifold has nontrivial second homology \cite{Co2}. Hence, the arguments in the disk case we used here do not work either. In order to prove
the surface case, we use a completely new approach. The main idea of the proof is as follows. First, we isometrically embed the original manifold $M$
into a larger manifold $\wh{M}$. Then, we utilize the fact that for any separating curve $\gamma$ in an absolutely area minimizing surface $\Sigma$
in $M$, $\gamma$ bounds a unique absolutely area minimizing surface $S\subset \Sigma$ in $M$ in the following way. For any simple closed curve
$\Gamma\subset\partial M$, consider a nearby simple closed curve $\wh{\Gamma}$ in $\wh{M}-M$. Then, if $\wh{\Sigma}$ is an absolutely area minimizing
surface in $\wh{M}$ with $\partial \wh{\Sigma}=\wh{\Gamma}$, then the curve $\Gamma'=\wh{\Sigma}\cap \partial M$ will be a uniqueness curve in
$\partial M$ near $\Gamma$. This shows density in the surface case. After this density result, we could adapt the summation argument in \cite{Co1} to
finish the proof of Theorem~\ref{surface}.

Note that this imbedding into a larger manifold argument can easily be adapted to the disk case reproving all the results in Section $3$, hence the
results in \cite{Co1}, too. Note also that the mean convexity of $M$ is very crucial to employ this approach (See Remark \ref{meanconvexrem}).

\vspace{.1in}

The organization of the paper is as follows: In the next section we cover some basic results
which will be used later. Section~3 contains the proof of the first main result of
the paper. In section 4, we prove the analogous result regarding absolutely area minimizing surfaces. Section~5 is devoted to
further remarks.

\section{Preliminaries}

In this section, we review the basic results which will be used in the following sections.

\begin{defn} \label{meanconvex} Let $M$ be a compact Riemannian $3$-manifold with boundary. Then $M$ is called {\em mean convex} (or {\em sufficiently convex})
if the following conditions hold.

\begin{itemize}

\item $\partial M$ is piecewise smooth.

\item Each smooth subsurface of $\partial M$ has nonnegative curvature with respect to an inward normal.

\item There exists a Riemannian manifold $N$ such that $M$ is isometric to a submanifold of $N$ and
each smooth subsurface $S$ of $\partial M$  extends to a smooth embedded surface $S'$ in $N$ such that $S' \cap M = S$.

\end{itemize}
\end{defn}

We call a simple closed curve {\em extreme} if it is on the boundary of its convex hull. Our results apply to the extreme curves as the convex hull
naturally satisfies the conditions above. Note that a simple closed curve in the boundary of a mean convex manifold $M$  is called as {\em weak
extreme} or {\em $H$-extreme curve}.

\begin{defn} An {\em area minimizing disk} is a disk which has the smallest area among disks with the same boundary.
An {\em absolutely area minimizing surface} is a surface which has the smallest area among all orientable surfaces (with no topological restriction)
with the same boundary.
\end{defn}

Now, we state the main facts which we use in the following sections.

\begin{lem}[\cite{MY2}, \cite{MY3}]\label{disjointdisks}
Let $M$ be a compact, mean convex $3$-manifold, and $\Gamma\subset\partial M$ be a  simple closed curve nullhomotopic in $M$. Then, there exists an
area minimizing disk $ D\subset M$ with $\partial  D = \Gamma$. All such disks are properly embedded in $M$, i.e. their boundaries are in $\partial
M$, and they are pairwise disjoint. Moreover, area minimizing disks spanning disjoint simple closed curves in $\partial M$ are also disjoint.
\end{lem}

Note that the last sentence in the previous theorem is known as Meeks-Yau Exchange roundoff trick. The main idea is as follows: If two area
minimizing disks $D_1$ and $D_2$ with disjoint boundaries intersect, the intersection will contain a closed curve $\beta$, and let $D_i^\beta \subset
D_i$ be the smaller disks bounded by $\beta$. Then, by swaping $D_1^\beta$ and $D_2^\beta$, we get a new area minimizing disk
$D_1'=(D_1-D_1^\beta)\cup D_2^\beta$ with a folding curve $\beta$. By pushing $D_1'$ along the folding curve $\beta$ to the convex side decreases
area which contradicts with $D_1'$ being area minimizing.

An analogous statement for absolutely area minimizing surfaces is obtained by combining the following results.

\begin{thm}[\cite{FF},\cite{ASS}, \cite{H}\label{absmincur}]
Let $M$ be a compact, strictly mean convex $3$-manifold and $\Gamma\subset\partial M$ a nullhomologous simple closed curve. Then there exists
$\Sigma\subset M$ an absolutely area minimizing surface with $\partial \Sigma = \Gamma$ and each such $\Sigma$ is smooth away from its boundary and
it is smooth around points of the boundary where $\Gamma$ is smooth.
\end{thm}

Hass proved the following statement for closed $3$-manifolds. It can be generalized with a slight modification of his argument. This lemma can be
considered as the adaptation of Meeks-Yau Exchange Roundoff trick to the surface case.

\begin{lem}[\cite{Ha}]\label{Hass} Let $M$ be an orientable, mean convex $3$-manifold, and $\Sigma_1$ and $\Sigma_2$ be two homologous, properly embedded,
absolutely area minimizing surfaces in $M$. If  $\partial \Sigma_1$ and $\partial \Sigma_2$ are disjoint or the same, then $\Sigma_1$ and $\Sigma_2$
are disjoint.
\end{lem}

\begin{proof} Since
$\Sigma_1$ and $\Sigma_2$ are in the same homology class, they separate a codimension-$0$ submanifold $M'$ from $M$, and $\Sigma_1\cup\Sigma_2
\subset \partial M'$. Then, $\Sigma_1$ and $\Sigma_2$ separate each other \cite{Ha}. Let $\Sigma_1 \setminus \Sigma_2 = S_1^+\cup S_1^-$, and
$\Sigma_2 \setminus \Sigma_1 = S_2^+\cup S_2^-$. Assuming $\partial S_1^-=\partial S_2^- = \Sigma_1\cap\Sigma_2$ ($S_1^+$ and $S_2^+$ are the components
containing $\partial \Sigma_1$ and $\partial \Sigma_2$ respectively), $\Sigma_1'=(\Sigma_1 \setminus S_1^-)\cup S_2^-$ would be another absolutely area minimizing surface in $M$
with boundary $\partial \Sigma_1$. This is because $\Sigma_1$ and $\Sigma_2$ are absolutely area minimizing surfaces, and $\partial S_1^-=\partial S_2^-$
implies $|S_1^-|=|S_2^-|$. However, $\Sigma_1'$ has singularity along $\Sigma_1\cap\Sigma_2$ which contradicts the regularity theorem for
absolutely area minimizing surfaces \cite{Fe}.
\end{proof}

Now, we state a lemma about the limit of area minimizing disks in a mean convex manifold.

\begin{lem}[\cite{HS}]\label{seq}
Let $M$ be a compact, mean convex $3$-manifold and let $\{ D_i\}$ be a sequence of properly embedded area minimizing disks in $M$. Then there is a
subsequence of $\{ D_i\}$ which converges to a countable collection of properly embedded area minimizing
disks in $M$.
\end{lem}

\section{Uniqueness of area minimizing disks}

This section is devoted to the proof of the following theorem.

\begin{thm}\label{disk}
Suppose that $M$ is a compact, orientable, mean convex $3$-manifold. Let $\N$ be the set of simple closed curves on the boundary of $M$ which are
nullhomotopic in $M$, and $\U \subset \N$ consist of those which bound unique area minimizing disks in $M$. Then $\U$ is not only dense but also a
countable intersection of open dense subsets of $\N$ with respect to the $C^0$-topology.
\end{thm}

\begin{rmk} \label{bdry} In the proof of this theorem we ignore the curves in $\N$ which bound area minimizing disks in $\partial M$. This is justified by
the fact that a curve $\gamma$ in the interior of an area minimizing disk $D \in \partial M$ cannot bound a properly embedded area minimizing
disk $D'$ since swapping the disk in $D$ bounded by $\gamma$ with $D'$ and rounding off (exchange-roundoff trick) would give a disk
with boundary the same as but area strictly smaller than $D$. In particular, such a curve $\gamma$ is clearly an interior point of $\U$.
\end{rmk}

\begin{proof}
For each $\Gamma \in \N$ fix an annulus neighborhood $A_{\Gamma} \subset \partial M$ and a properly embedded annulus $A'_{\Gamma} \subset M$ with
$\partial A_{\Gamma} = \partial A'_{\Gamma}$ as in Lemma~\ref{keylem}, i.e.,

\begin{itemize}
\item $A_\Gamma \cup A'_\Gamma$ bounds a solid torus in $M$, and
\item if the boundary of a properly embedded area minimizing disk $D\subset M$ is essential in $A_{\Gamma}$, then $D$ intersect
$A'_{\Gamma}$ in a unique essential simple closed curve (see Figure~\ref{lens}).
\end{itemize}

\begin{figure}[b]\label{lens}

\relabelbox  {\epsfxsize=2in

\centerline{\epsfbox{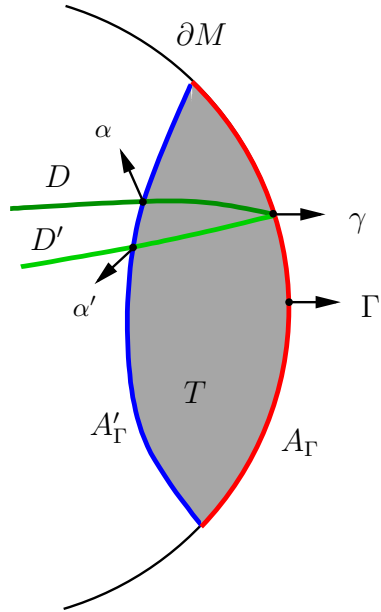}}}

\relabel{1}{$\partial M$}

\relabel{2}{$D$}

\relabel{3}{ $D'$}

\relabel{4}{ $A_\Gamma'$}

\relabel{5}{ $\gamma$}

\relabel{6}{ $\Gamma$}

\relabel{7}{$A_\Gamma$}

\relabel{8}{$T$}

\relabel{9}{\footnotesize $\alpha'$}

\relabel{10}{\footnotesize $\alpha$}

\endrelabelbox

\caption{\label{fig1} \small {For any $\gamma \subset A_\Gamma\subset \partial M$, any area minimizing disk $D$ with $\partial D=\gamma$
intersects $A_\Gamma'$ in a unique essential curve $\alpha$. The grey region represents the solid torus $T$ in $M$ with $\partial T=A_\Gamma\cup
A_\Gamma'$.}}

\end{figure}

For an essential simple closed curve $\gamma$ in $A_\Gamma$, let $R^\Gamma_\gamma$ denote (as in Lemma~\ref{region2}) the smallest annulus in
$A'_\Gamma$ which contains the intersection of $A'_\Gamma$ with all the area minimizing disks spanning $\gamma$. Note that $\gamma \in \U$ if and
only if $|R^\Gamma_\gamma| = 0$,  where $| \cdot | $ denotes the area.

First we will prove that $\U$ is dense in $\N$. Let $\Gamma \in \N$, and foliate $A_\Gamma$  by essential simple closed curves $\{ \Gamma_t : t
\in [-\epsilon, \epsilon]\}$ such that $\Gamma_0 = \Gamma$. By Lemma~\ref{region1}, the regions $R^\Gamma_{\Gamma_t}$ and $R^\Gamma_{\Gamma_s}$
in $A'_\Gamma$ are disjoint for $s \neq t$. Therefore $$\sum_{t \in [-\epsilon, \epsilon]} |R_{\Gamma_t}^\Gamma | < |A'_\Gamma| < \infty \ .
$$ Hence $| R_{\Gamma_t}^\Gamma| > 0$ only for countably many $ t \in [-\epsilon , \epsilon]$, i.e. $\Gamma_t$ bounds a unique area minimizing
disk for uncountably many $t\in [-\epsilon , \epsilon]$. Since we began with an arbitrary $\Gamma \in \N$, this proves that $\U$ is dense in
$\N$.

To prove that $\U$ is the intersection of countably many open dense subsets of $\N$ let
$$U_n = \{ \gamma \in \N | \ \mbox{there exists } \
\Gamma \in \N \mbox{ such that } \gamma \mbox{ is essential in } A_\Gamma \mbox { and } |R_\gamma^\Gamma| < 1/n \} \ ,$$ for every $n \in \Z_+$.
Observe that $\U = \cap_{n\in \mathbb{N}} U_n$, and in particular, each $U_n$ is dense. It remains to show that every $U_n$ is open. Let $\gamma \in
U_n$, choose $\Gamma \in \N$ such that $\gamma$ is essential in $A_\Gamma$ with $| R_\gamma^\Gamma | < 1/n$, and choose an annular region $R$ in
$A'_\Gamma$  with $|R| < 1/n$ whose interior contains $R^\Gamma_\gamma$. Since $\U$ is dense in $\N$, there is a sequence $\{ \gamma_n \}$ of
pairwise disjoint, essential curves in $A_\Gamma$ converging to $\gamma$ such that each $\gamma_n$ bounds a unique area minimizing disk $D_n$ in $M$.
We can arrange that all these curves are in a prescribed component of $A_\Gamma \setminus \gamma$. By Lemma~\ref{seq}, the sequence $\{ D_n \}$ has a
subsequence converging to a countable collection of area minimizing disks spanning $\gamma$. This implies the existence of essential curves
$\gamma^+$ and $\gamma^-$ in $A_\Gamma$ such that

\begin{itemize}
\item the curves $\gamma^+$ and $\gamma^-$ are contained in different components of $A_\Gamma \setminus \gamma$,
\item each of $\gamma^{\pm}$ bounds a unique area minimizing disk $D^{\pm}$, and
\item $D^{\pm} \cap A'_\Gamma \subset R$.
\end{itemize}

Let $A_\gamma$ be the open annulus in $A_\Gamma$ bounded by $\gamma^{\pm}$, and $V_\gamma$ be the set of all simple closed curves essential in
$A_\gamma$. Note that $V_\gamma$ is an open neighborhood of $\gamma$ in $\N$. Moreover, $V_\gamma \subset U_n$ because $D^+ \cup D^-$ separates
the solid torus bounded by $A_\Gamma \cup A'_\Gamma$, and an area minimizing disk spanning any $\alpha \in V_\gamma$ has to be disjoint from
$D^+ \cup D^-$, forcing $R_\alpha^\Gamma$ to remain inside $R$. This proves that $U_n$ is open in $\N$ and finishes the proof.
\end{proof}

In the rest of the section we will prove the lemmas used in the proof of Theorem~\ref{disk}.

\begin{lem}\label{keylem}
For every $\Gamma \in \N$, there exist annuli $A_\Gamma$ and $A'_\Gamma$ with common boundary, former a neighborhood of $\Gamma$ in $\partial M$
and latter properly embedded in $M$, such that $A_\Gamma \cup A'_\Gamma$ bounds a solid torus in $M$, and any properly embedded area minimizing
disk in $M$ spanning an essential curve in $A_\Gamma$ intersects $A'_\Gamma$ in a unique essential curve.
\end{lem}

\begin{proof}
Given $\Gamma \in \N$, we choose an annulus neighborhood $A_\Gamma$ and a solid torus neighborhood $N_\Gamma \supset A_\Gamma$ of $\Gamma$ in
$\partial M$ and $M$, respectively. We may make the initial annulus neighborhood smaller as we proceed, but will keep denoting it by $A_\Gamma$,
abusing the notation. Note that, by \cite{MY2}, we can choose $A_\Gamma$ sufficiently small so that there is an area minimizing annulus $A$
in $M$ with boundary $\partial A_\Gamma$. If there is such a {\it properly embedded} area minimizing annulus $A$, then let $A'_\Gamma$ be $A$.
Otherwise $A_\Gamma$ is the unique area minimizing annulus with boundary $\partial A_\Gamma$, and we will now explain how to construct
$A'_\Gamma$ in this case.

Let $\Gamma^+$ and $\Gamma^-$ denote the boundary components of $A_\Gamma$, $D^{\pm}$ be area minimizing disks spanning $\Gamma^{\pm}$, and $\{
\gamma^{\pm}_n\}$ be sequences of disjoint simple closed curves in the interior of $D^{\pm}$ converging to $\Gamma^{\pm}$. Let $\widehat{M}$ be
the component of $\overline{ M \setminus (D^+ \cup D^-)}$ which contains $\Gamma$. Note that $\widehat{M}$ is mean convex as $D^{\pm}$ are
minimal, and $\gamma^{\pm}_n$ can be considered as simple closed curves in $\partial \widehat{M}$. Therefore, by choosing $A_\Gamma$
sufficiently small and $n$ sufficiently large, we can guarantee that there is an area minimizing annulus $A_n$ in $\widehat{M}$ spanning
$\gamma_n^+ \cup \gamma_n^-$. Let $A'_\Gamma$ to be the union of $A_n$ and the obvious (area minimizing) annuli in $D_{\pm}$ between
$\gamma^{\pm}_n $ and $ \Gamma^{\pm}$.

Before we proceed, we will prove that for sufficiently small $A_\Gamma$ and sufficiently large $n$, $A_n$ is an area minimizing surface not only
in $\widehat{M}$ but also in $M$. Assume that there is an annulus $A'_n \subset M$ such that $\partial A'_n = \partial A_n= \gamma^+_n \cup
\gamma^-_n$ and $|A'_n| < |A_n|$. Since $A_n$ is area minimizing in $\widehat{M}$, $A'_n$ cannot be embedded in $\widehat{M}$. Without loss of
generality, assume that $A'_n \cap (D^+ \setminus \gamma^+_n) \neq \emptyset$. Any component $\alpha$ of $A'_n \cap D^+$ has to be essential in
$A'_n$, since otherwise we would swap the disks bounded by $\alpha$ in $A'_n$ and in $D^+$ to get a contradiction, using the exchange-roundoff
trick. If a component $\alpha$ of $A'_n \cap D^+$ and $\gamma_n^+$ are concentric in $D^+$, then we get a contradiction (again by the
exchange roundoff trick) by swapping the annular regions between $\gamma^+_n$ and $\alpha$ in $D^+$ and in $A'_n$. Therefore any component
$\alpha$ of $A'_n \cap D^+$ has to be essential in $A'_n$ and nullhomotopic in $D^+ \setminus D_n$, where $D_n$ denotes the disk in $D^+$
bounded by $\gamma_n^+$. Consider the annulus $A''_n$ in $A'_n$ with $\partial A''_n = \alpha \cup \gamma^+_n$ and the disk $D_\alpha$ in $D^+$
bounded by $\alpha$. Note that the disk $D=A''_n \cup D_\alpha$ bounds $\gamma^+_n$ hence $|D_n| \leq |D| \leq |A''_n| +|D_\alpha|$. But the
facts that $D_\alpha$ is a subset of $D^+ \setminus D_n$ and the sequence $\{\gamma_n^+ =
\partial D_n \}_n$ converges to $\Gamma^+ = \partial D^+$ imply that $|D_\alpha|$ can be made arbitrarily small. Hence to get a contradiction,
all we need to do is make $A_\Gamma$ sufficiently small and $n$ sufficiently large, forcing $|A''_n| +|D_\alpha| < |D_n|$.

Now we have defined $A'_\Gamma$  regardless of whether $A_\Gamma$ bounds a properly embedded area minimizing annulus in $M$ or not. Note that
$A'_\Gamma$ is properly embedded in $M$, $\partial A'_\Gamma = \partial A_\Gamma$, and $A'_\Gamma \cup A_\Gamma$ bounds a solid torus $T$ in $M$
(at least when we choose $A_\Gamma$ small enough to make sure that $A'_\Gamma$ remains in the solid torus neighborhood $N_\Gamma$ of $\Gamma$).
Also note that $A'_\Gamma$ is either area minimizing or it is the union of three area minimizing annuli glued along $\gamma^{\pm}_n$.

In the rest of the proof, we will show that for any properly embedded area minimizing disk $D_\gamma$ spanning an essential curve $\gamma$ in
$A_\Gamma$, $D_\gamma \cap A'_\Gamma$ is a unique essential curve in $A'_\Gamma$: First of all, since $\gamma$ is essential in $A_\Gamma$, it is also
essential in the solid torus $T$ and cannot bound any surface in $T$. Therefore $D_\gamma$ has to intersect $A'_\Gamma$. Moreover, any component
$\alpha$ of $D_\gamma \cap A'_\Gamma$ has to be an essential curve in $A'_\Gamma$ since otherwise we would swap the disks bounded by $\alpha$ in
$D_\gamma$ and in $A'_\Gamma$ to get a contradiction using the exchange-roundoff trick.

Now, assume that $D_\gamma \cap A'_\Gamma$ has two components $\alpha_1$ and $\alpha_2$. These curves cannot be concentric in $D_\gamma$ since
otherwise, again by using the exchange-roundoff trick, we would get a contradiction with the area minimizing property of $D_\gamma$
after swapping the annular regions between the $\alpha_i$'s in $D_\gamma$ and in $A'_\Gamma$. We eliminate the remaining possibility of
nonconcentric $\alpha_i$'s by choosing $A_\Gamma$ with sufficiently small area compared to that of an area minimizing disk $D_\Gamma$ spanning
$\Gamma$. Let $\alpha$ be any component of $D_\gamma \cap A'_\Gamma$ and $D_\alpha$ be the disk in $D_\gamma$ bounded by $\alpha$. We have the
following inequalities by area minimizing properties of $D_\gamma$, $D_\Gamma$, $D^+$, and that of $A'_\Gamma$ (or, depending on the
construction of $A'_\Gamma$, $A_\Gamma$ and $A_n$, and the convergence of $\{\gamma^+_n\}$ to $\Gamma^+$):
$$ |D_\gamma| + |A_\Gamma| > |D^+| \ , \ |D^+|+|A_\Gamma| > |D_\gamma| \ , $$
$$ |D_\Gamma| + |A_\Gamma| > |D^+| \ , \  |D^+|+|A_\Gamma| > |D_\Gamma| \ , $$
$$ | D_\alpha| + |A'_\Gamma| > |D^+| \ , \  |A_\Gamma| \geq |A'_\Gamma| \ . $$
 It follows that

$$
|D_\gamma \setminus D_\alpha| =
|D_\gamma| - |D_\alpha| < |D^+| + |A_\Gamma| - |D_\alpha|
< |A'_\Gamma| + |A_\Gamma| \leq 2 |A_\Gamma| \ .
$$

Assuming that the components $\alpha_1$ and $\alpha_2$ of $D_\gamma \cap A'_\Gamma$ are not concentric in $D_\gamma$, we get
$$ |D_\gamma| = |(D_\gamma \setminus D_{\alpha_1}) \cup (D_\gamma \setminus D_{\alpha_2}) | <
|(D_\gamma \setminus D_{\alpha_1})| + |(D_\gamma \setminus D_{\alpha_2}) | < 4 |A_\Gamma| \ . $$
Hence
$$ | D_\gamma| > |D^+| - |A_\Gamma| > |D_\Gamma| - 2 |A_\Gamma|$$
leads to
$$ |D_\Gamma| < 6 |A_\Gamma|$$
which is impossible once we choose $|A_\Gamma|$ sufficiently small since $|D_\Gamma|$ is independent of this choice.
\end{proof}

In the following lemmas, we have an arbitrary $\Gamma \in \N$ and fixed annuli $A_\Gamma$ and $A'_\Gamma$ as in Lemma~\ref{keylem}.

\begin{lem}\label{region1}
Let $\gamma$ and $\gamma'$ be disjoint, essential simple closed curves in $A_\Gamma$, $D_1$ and $D_2$ be distinct properly embedded area
minimizing disks in $M$ bounding $\gamma$, $\alpha_i = D_i \cap A'_\Gamma$, and $R \subset A'_\Gamma$ be the annulus bounded by $\alpha_1$ and $
\alpha_2$. Then any area minimizing disk in $M$ spanning $\gamma'$ is disjoint from $R$.
\end{lem}

\begin{proof}
Observe that each of the disks $D_1$ and $D_2$ separates the solid torus $T$ with $\partial T = A_\Gamma \cup_{\Gamma^{\pm}} A'_\Gamma$ into two
pieces. Since $D_1 \cap D_2 = \gamma \subset \partial T $, $D_1 \cup_{\gamma} D_2$ separates $T$ into three pieces, and $R$ is ``half" (the annulus
$(D_1 \cup_{\gamma} D_2) \cap T$ being the other ``half") of the boundary of the ``middle" piece $T_0$. Note that $T_0 \cap A_\Gamma = \gamma$,
therefore $\gamma'$ does not intersect $T_0$. If an area minimizing disk spanning $\gamma'$ were to intersect $R$, this would force it to intersect
either $D_1$ or $D_2$, but this is impossible since properly embedded area minimizing disks with disjoint boundaries do not intersect by Lemma
\ref{disjointdisks}.
\end{proof}

\begin{lem}\label{region2}
For every simple closed curve $\gamma$ which is essential in $A_\Gamma$ and bounds a properly embedded area minimizing disk in $M$ there is a
subset $R_\gamma^\Gamma$ of $A'_\Gamma$ such that
\begin{enumerate}
\item the intersection of $A'_\Gamma$ and any area minimizing disk spanning $\gamma$ belongs to $R_\gamma^\Gamma$,
\item $R_\gamma^\Gamma$ is an annulus if $\gamma \notin \U$,
\item $R_\gamma^\Gamma$ is a simple closed curve if $\gamma \in \U$, and
\item if $\gamma$ and $\gamma'$ are disjoint, so are $R_\gamma^\Gamma$ and $R^\Gamma_{\gamma'}$.
\end{enumerate}
\end{lem}

\begin{proof}
If $\gamma \in \U$, then the definition of $R_\gamma^\Gamma$ is obvious and (3) is a consequence of Lemma~\ref{keylem}. Assume that, $\gamma
\notin \U$, and consider all the curves obtained as the intersection of $A'_\Gamma$ with an area minimizing disk spanning $\gamma$. Let
$R_\gamma^\Gamma$ be the union of all the annuli bounded by any pair of such curves. (1) and (2) hold by definition and connectedness of
$R_\gamma^\Gamma$. (4) is a consequence of Lemma~\ref{region1}.

\end{proof}

\section{Uniqueness of absolutely area minimizing surfaces}

This section is devoted to the proof of the following theorem.

\begin{thm}\label{surface}
Suppose that $M$ is a compact, orientable, mean convex $3$-manifold. Let $\B$ be the set of simple closed curves on the boundary of $M$ which are
nullhomologous in $M$, and $\T \subset \B$ consist of those which bound a unique absolutely area minimizing surface in $M$. Then, $\T$ is not only
dense but also countable intersection of open dense subsets of $\B$ with respect to the $C^0$-topology.
\end{thm}

\begin{rmk} \label{intersection} In order to prove this theorem, one might want to employ a similar method to the disk case. However, the crucial step in this method is Lemma
\ref{disjointdisks}, i.e. the area minimizing disks $D_1$ and $D_2$ in $M$ bounding the simple closed curves $\Gamma_1$ and $\Gamma_2$ in $\partial
M$ are disjoint provided that the curves $\Gamma_1$ and $\Gamma_2$ are disjoint, and this is not true for the absolutely area minimizing surfaces
case. There exist disjoint $H$-extreme curves which bound intersecting absolutely area minimizing surfaces  \cite{Co2}.
\end{rmk}

\begin{rmk} \label{bdry2} Like in the disk case, we will ignore the curves in $\B$ which bound absolutely area minimizing surfaces in $\partial M$ (See Remark \ref{bdry}).
\end{rmk}

\begin{prop}\label{density}  $\T$ is dense in $\B$ with respect to the $C^0$-topology.
\end{prop}
\begin{proof}
Assume otherwise. Then there is a simple closed curve $\Gamma$ with a neighborhood $N_\epsilon(\Gamma)$ in $\partial M$ such that any simple
closed curve $\Gamma' \subset N_\epsilon(\Gamma)$, i.e. $d(\Gamma, \Gamma')<\epsilon$ in $C^0$-metric, bounds at least two absolutely area minimizing
surfaces $\Sigma_1'$ and $\Sigma_2'$ in $M$.

This implies that an absolutely area minimizing surface $\Sigma$ in $M$ with $\partial \Sigma = \Gamma$ cannot lie in $\partial M$.
Indeed, since $M$ is mean convex, by the maximum principle, $\Sigma\cap \partial M = \Gamma$. This is because if $\Sigma\subset \partial M$, then for any
simple closed curve $\alpha$ near $\Gamma$ in $\Sigma \subset \partial M$, $\alpha$ must bound a unique absolutely area minimizing surface.
Otherwise, if $\alpha$ bounds $\Sigma_1\subset \Sigma$ and another absolutely area minimizing surface $\Sigma_2$ in $M$, then
$\Sigma'=(\Sigma \setminus \Sigma_1) \cup \Sigma_2$ would be yet another absolutely area minimizing surface with boundary $\Gamma$ since $|\Sigma|=|\Sigma'|$. However,
there is a singularity along $\alpha$ in $\Sigma'$. This contradicts the regularity theorem for absolutely area minimizing surfaces \cite{Fe}.

Now, embed $M$ into a larger $3$-manifold $N$ isometrically as in Definition \ref{meanconvex}, i.e. $M$ is isometric to a codimension-$0$ submanifold
of $N$. We abuse the notation and use $M$ to denote this submanifold. For every $\delta>0$, let $M_\delta$ denote the $\delta$-neighborhood of $M$ in
$N$.

For each $j \in \Z_+$, consider a sequence of curves $\{ \wh{\Gamma}_i^j\}_{i=1}^{\infty}$ in $M_{1/j} \setminus M$ which converges to $\Gamma$ as
$i$ tends to $\infty$. For every  $i, j \in \Z_+$, let $\wh{\Sigma}_i^j$ be an absolutely area minimizing surface in $M_{1/j}$ with $\partial
\wh{\Sigma}_i^j = \wh{\Gamma}_i^j$. For each $j$, by Federer's compactness theorem \cite{Fe}, a subsequence of $\{ \wh{\Sigma}_i^j\}_i$ converges to
an absolutely area minimizing surface $\Sigma^j$ in $M_{1/j}$ with $\partial \Sigma^j = \Gamma$. As a further consequence of compactness, the
sequence $\{ \Sigma^j \}_{j=1}^{\infty}$ has a subsequence converging to an absolutely area minimizing surface $\Sigma$ in $M$ with $\partial \Sigma
= \Gamma$.

\vspace{.2cm}

\noindent {\bf Claim:} There exists $j \in \Z_+$ such that $\Sigma^j \subset M$, and hence $\Sigma^j$ is an absolutely area minimizing surface in
$M_{1/k}$ for every $k \geq j$.

\noindent {\em Proof of the Claim:} Assume that $\Sigma^j \setminus M\neq \emptyset$  for all $j$. Now, replace the sequence $\Sigma^j$ with the
sequence $\Sigma_M^j=\overline{\Sigma^j\cap int(M)}$ which also converges to $\Sigma$. Since $int(\Sigma)\cap\partial M =\emptyset$ by assumption, we
can assume that $\Sigma_M^j$ is connected by ignoring the smaller pieces if necessary. Now, consider $\Gamma^j = \partial \Sigma_M^j$ in $\partial
M$. If $\Gamma^j = \Gamma$ for infinitely many $j$, then a sequence of interior points of $\Sigma^j$'s would converge to a point in $\partial M$,
contradicting the assumption that $int(\Sigma)\cap\partial M =\emptyset$. Therefore $\Gamma^j$ is distinct from $\Gamma$ (may intersect it) for all
but finitely many $j$. On the other hand, $\Gamma^j$ converges to $\Gamma$ since $\Sigma_M^j$ converges to $\Sigma$. Hence for sufficiently large
$j$, $\Gamma^j\subset N_\epsilon(\Gamma)$ and by assumption, $\Gamma^j$ bounds at least one absolutely area minimizing surfaces $S_2$ other than
$S_1=\Sigma_M^j$ in $M$ (see Figure~\ref{figure2}). By swaping $S_1$ and $S_2$ in $\Sigma^j$, we get a new surface $\tilde{\Sigma}^j=(\Sigma^j
\setminus S_1)\cup S_2$ which has the same area as $\Sigma^j$. Hence, $\tilde{\Sigma}^j$ is an absolutely area minimizing surface in $M_{1/j}$ with
boundary $\Gamma$. However, $\tilde{\Sigma}^j$ is singular along $\Gamma^j$ which contradicts the regularity theorem for absolutely area minimizing
surfaces \cite{Fe}. This finishes the proof of the claim.

\begin{figure}[b]\label{figure2}
\begin{center}

\relabelbox {\epsfxsize=2.8in \centerline{\epsfbox{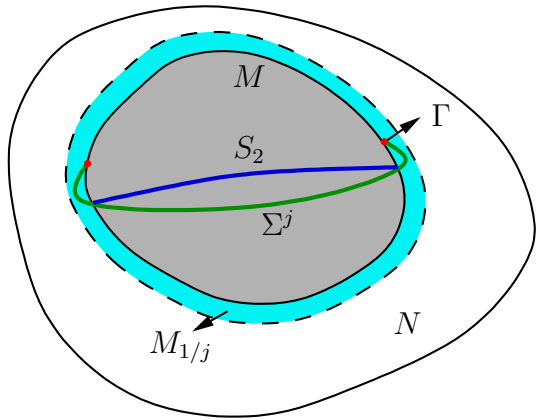}}} \relabel{1}{$N$} \relabel{2}{$M$} \relabel{3}{$M_{1/j}$} \relabel{4}{$S_2$}
\relabel{5}{$\Sigma^j$} \relabel{6}{$\Gamma$} \endrelabelbox

\end{center}

\caption{\label{fig23} \footnotesize $\Sigma^j$ is the absolutely area minimizing surface in $M_{1/j}$ with $\partial \Sigma^j=\Gamma$. $S_2$ is
another absolutely area minimizing surface with $\partial S_2 = \partial (\Sigma^j\cap M)$.}
\end{figure}

Now, to finish the proof of the proposition, we get a contradiction as follows. By using the claim above, fix a positive integer $j$ such that
$\Sigma^j \subset M$. Then $\Sigma^j$ is an absolutely area minimizing surface in $M$ with $\partial \Sigma^j = \Gamma$. Let $\Sigma_i =
\wh{\Sigma}^j_i \cap M$, and $\Gamma_i = \partial \Sigma_i = \wh{\Sigma}_i^j \cap
\partial M$. $\Sigma_i$ converges to $\Sigma^j$ and $\Gamma_i$ converges to $\Gamma$ since $\wh{\Sigma}_i^j $ converges to $\Sigma^j$ (as $i$
approaches to $\infty$). Therefore for sufficiently large $i_o$, $\Gamma_{i_o} \subset N_{\epsilon}(\Gamma)$, and consequently, $\Gamma_{i_o}$ bounds
at least one other absolutely area minimizing surface $\Sigma'_{i_o}$ in $M$ beside $\Sigma_{i_o}$ by the assumption.

Let $\tilde{\Sigma}_{i_o}^j = (\wh{\Sigma}^j_{i_o} \setminus \Sigma_{i_o}) \cup \Sigma'_{i_o}$. Since $\tilde{\Sigma}_{i_o}^j$ has the same area and
boundary as $\wh{\Sigma}_{i_o}^j$, it is also an absolutely area minimizing surface in $M_{1/j}$. However, it is singular along $\Gamma_{i_o}$
contradicting the regularity theorem for absolutely area minimizing surfaces \cite{Fe}.
\end{proof}

\begin{rmk} \label{meanconvexrem} The mean convexity of $M$ is very crucial in the proof above. If $M$ was not mean convex, then it is easy to construct
examples where for any $j \in \Z_+$, the absolutely area minimizing surface $\Sigma^j \subset M_{\frac{1}{j}}$ with $\partial \Sigma_j = \Gamma$, and
$\Sigma_j\subset M_{\frac{1}{j}}-M$. One can simply take a $3$-manifold $M$ which is not mean convex, and the absolutely area minimizing surface
$\Sigma$ with boundary $\Gamma\subset \partial M$ completely lies in $\partial M$, i.e. $\Sigma\subset \partial M$. Then, for such a manifold
$\wh{\Sigma}_i^j \cap M$ might be empty for any $i$, and the whole argument collapses (See also Remark \ref{bdry2}).
\end{rmk}

\begin{prop}\label{generic}  $\T$ is a countable intersection of open dense subsets of $\B$ with respect to the $C^0$-topology.
\end{prop}
\begin{proof} Let $\Gamma\in \T$ be a uniqueness curve, i.e. $\Gamma\subset \partial M$ bounds a unique absolutely area minimizing surface $\Sigma$ in $M$. Let
$\{\Gamma^+_i\}$ be a sequence of pairwise disjoint simple closed curves in $\T$ which converges to $\Gamma$. We also assume that every $\Gamma^+_i$ is on
the same (say positive) side of $\Gamma$, i.e. $A^+_i\subset A^+_j$ when $i>j$, where $A^+_i=[\Gamma,\Gamma^+_i]$ is the annulus component of
$\partial M \setminus (\Gamma \cup \Gamma^+_i)$ for any $i$.

For each $i$, there exists a unique absolutely area minimizing surface $\Sigma^+_i$ in $M$ with
$\partial \Sigma^+_i = \Gamma^+_i$. By compactness theorem, a subsequence of $\{\Sigma^+_i\}$, which will also be denoted by $\{\Sigma^+_i\}$ by
abusing the notation, converges to $\Sigma$ which is the unique absolutely area minimizing surface
in $M$ with  boundary $\Gamma$.

Take a tubular neighborhood $N(\Sigma) \simeq \Sigma \times (-1,1)$ of $\Sigma$ in $M$. Since $\Sigma^+_i$ converges to $\Sigma$, there exists an
$N_0$ such that for any $i\geq N_0$, $\Sigma^+_i\subset N(\Sigma)$ and $\Gamma^+_i$ is isotopic to $\Gamma$ in $\partial N(\Sigma)$. Unlike the disk
case, {\em a priori} we do not know that$\Sigma^+_i\cap\Sigma = \emptyset$ even when $\Gamma^+_i\cap\Gamma = \emptyset$ (see Remark
\ref{intersection}). However, since $\Gamma^+_i$ separates the annulus $\partial N(\Sigma)$ for $i \geq N_0$, $\Sigma^+_i$ separates the product
neighborhood $N(\Sigma)$. Therefore, for $i\geq N_0$, $\Sigma^+_i$ is in the same homology class as $\Sigma$, and consequently, by Lemma \ref{Hass},
$\Sigma^+_i$ and $\Sigma$ are disjoint (See Figure \ref{fig45} left). Let us denote the component of $M \setminus (\Sigma \cup \Sigma_i^+)$ whose
boundary contains $A_i^+$ by $M_i^+ = [\Sigma,\Sigma^+_i]$.

\vspace{.2cm}

\noindent {\bf Claim:} There exists $N_1 \geq N_0$ such that for $i>N_1$, any absolutely area minimizing surface $S$ whose boundary is $C^0$-close
and isotopic to $\Gamma$ in $A^+_i $  is contained in $M^+_i$. Consequently,  $S$ is in the same homology class with $\Sigma$, by the arguments above
(See Figure \ref{fig45} right).

\noindent {\em Proof of the Claim:} Assume otherwise, i.e.,  for any $i>N_0$, we can find a sequence of absolutely area minimizing surfaces
$S_i$ in $M$ with $\partial S_i\subset A^+_i$ and $S_i \nsubseteq M^+_i$. If $S_i$ and  $\Sigma_{N_0}^+$ are disjoint, then $\Sigma$ separates
$S_i$ since $\Sigma^+_{N_0}\cup\Sigma$ separates $M$, but by using the swapping argument above, we get a new absolutely area minimizing surface
$S_i'$ with singularity along $S_i\cap\Sigma$ contradicting regularity. The assumption that $S_i$ is disjoint from $\Sigma$ leads to a similar
contradiction. Therefore we have a sequence of absolutely area minimizing surfaces $S_i $ in $M$ such that for every $i \geq N_0$, $S_i$
intersects both $\Sigma$ and $\Sigma^+_i$.

\begin{figure}[t]\label{fig45}
\begin{center}
$\begin{array}{c@{\hspace{.4in}}c}

\relabelbox {\epsfxsize=2.6in {\epsfbox{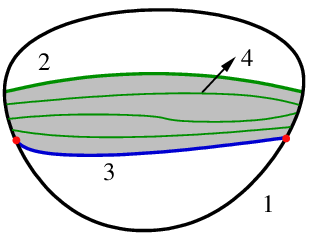}}} \relabel{1}{$M$} \relabel{2}{$\Sigma^+_{N_0}$} \relabel{3}{$\Sigma$} \relabel{4}{$\Sigma_i^+$}
\endrelabelbox  &

\relabelbox  {\epsfxsize=2.8in \epsfbox{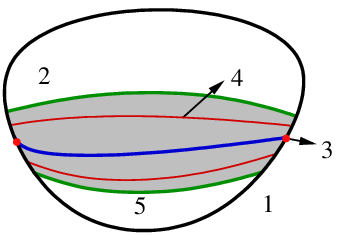}} \relabel{1}{$M$} \relabel{2}{$\Sigma^+_{N_1^+}$} \relabel{3}{$\Gamma$}
\relabel{4}{$S$} \relabel{5}{$\Sigma^-_{N_1^-}$} \endrelabelbox \\
\end{array}$

\end{center}
\caption{\label{fig45} \footnotesize On the left: for $i\geq N_0$, $\Sigma^+_i$ is in the same homology class as $\Sigma$, and consequently,
$\Sigma^+_i$ and $\Sigma$ are disjoint. On the right: for any absolutely area minimizing surface $S$ with $\partial S\subset
[\Gamma_{N_1^-},\Gamma_{N_1^+}]$ is in the same homology class with $\Sigma$, and hence any such $S$ and $S'$ are disjoint whenever $\partial
S\cap\partial S' =\emptyset$. }
\end{figure}

Since $\partial S_i$ converges to $\Gamma$, and $\Gamma$ is a uniqueness curve, by compactness theorem, after passing to a subsequence if necessary,
$S_i$ converges to $\Sigma$. However, since $S_i\cap\Sigma^+_{N_0}\neq \emptyset$ for any $i>N_0$, and $\Sigma^+_{N_0}$ is compact, the limit of the
sequence $ \{ S_i\}$ must have a limit point on $\Sigma^+_{N_0}$. Since $\Sigma^+_{N_0}\cap \Sigma=\emptyset$, this is a contradiction. The claim
follows.

Obviously, a similar statement holds for the ``negative side" of $\Gamma$. Therefore, every uniqueness curve $\Gamma$ in $\T$,
has a tubular neighborhood $A_\Gamma$ in $\partial M$ such that all absolutely area minimizing surfaces in $M$ with boundary isotopic to $\Gamma$ in $A_\Gamma$ are in the same homology class. In particular, any two distinct absolutely area minimizing surfaces with the same boundary in $A_\Gamma$ are disjoint by Lemma \ref{Hass}. Similarly, any two absolutely area minimizing surfaces with disjoint boundaries in $A_\Gamma$ are also disjoint.

Now, we will show that $\T$ is countable intersection of open dense subsets. We will follow the arguments in the main theorem of \cite{Co1}. Above,
we showed that for any simple closed curve $\Gamma$ in $\T$, there is a neighborhood $N_\Gamma$ (corresponding to the curves isotopic to $\Gamma$ in
$A_\Gamma$ above) in $C^0$ topology such that for any $\Gamma' \in N_\Gamma$, an absolutely area minimizing surface $S$ with $\partial S=\Gamma'$ is
in the same homology class with $\Sigma$, where $\Sigma$ is the unique absolutely area minimizing surface in $M$ with $\partial \Sigma=\Gamma$. This
implies that any two absolutely area minimizing surfaces with disjoint or matching boundaries in $N_\Gamma$ must be disjoint. Now, let $\C=
\bigcup_{\Gamma\in\T} N_\Gamma$. As $\T$ is dense in $\B$ by Proposition~\ref{density}, $\C$ is open dense in $\B$.

The rest of the proof is along the same lines as the proof of Theorem 3.2 in \cite{Co1}, more precisely the part regarding Claim 2. Here we give an
outline and refer the reader to \cite{Co1} for further details. For each $\alpha\in\C$, we can construct a canonical neighborhood $\Omega_\alpha
=[\Sigma^-_\alpha, \Sigma^+_\alpha]$, (the region between ``extremal" absolutely area minimizing surfaces $\Sigma^-_\alpha$ and $\Sigma^+_\alpha$
with $\partial \Sigma^\pm_\alpha =\alpha$) which contains every absolutely area minimizing surface in $M$ with boundary $\alpha$. By construction,
$\Omega_\alpha$ is independent of $N_\Gamma$ and depends only on $\alpha$. By the disjointness of absolutely area minimizing surfaces with boundary
in $\C$, if $\alpha\cap\beta=\emptyset$, then $\Omega_\alpha\cap\Omega_\beta=\emptyset$. Also, if $\alpha$ is a uniqueness curve, then
$\Sigma^+_\alpha=\Sigma^-_\alpha$ and $\Omega_\alpha=\Sigma_\alpha^{\pm}$ should be considered as a {\em degenerate} region (with no thickness).

Let $s_\alpha$ be the volume of $\Omega_\alpha$ and define $U_i=\{\ \alpha\in \C \ | \ s_\alpha<\frac{1}{i} \ \}$ for each $i \in \Z_+$. Note that $\T$ is contained in every $U_i$  since $s_\alpha = 0$ for every $\alpha\in\T$, by definition. In particular, $U_i$ is dense in $\B$. Moreover, $\T=\bigcap_{i=1}^{\infty} U_i$, by construction. Finally, by using the arguments similar to those in the proof of Theorem 3.2 in \cite{Co1}, one can prove that $U_i$ is open in $\C$, hence in $\B$.
\end{proof}

\begin{rmk} Notice that in the proof of Proposition~\ref{generic}, we show that for any simple closed curve $\Gamma\in \T$, there exists an annular neighborhood
$A_\Gamma$ of $\Gamma$ in $\partial M$, such that any absolutely area minimizing surface with boundary in $A_\Gamma$ must be in the same homology
class as the unique absolutely area minimizing surface with boundary $\Gamma$ (See Figure \ref{fig45} right). This is interesting in its own right,
and shows local constancy of the homology classes of absolutely area minimizing surfaces in some sense.

\end{rmk}

\section{Further remarks}

The density and genericity results in Section $3$ and $4$ are about $C^0$ simple closed curves in $\partial M$ with $C^0$-topology. Note that the
arguments in these results easily generalize to the smooth case. In particular, let $\N^k$ be the set of $C^k$ simple closed curves in $\partial M$ which are nullhomotopic in $M$. Then Theorem~\ref{disk}  generalizes to $\U^k=\U\cap \N^k$ in
$C^0$-topology. Moreover, this implies that if $\partial M$ smooth, then $\U^\infty$ is dense in $\N$ in $C^0$-topology. In other
words, when $\partial M$ is smooth, then for any $C^0$ nullhomotopic simple closed curve $\Gamma$ in $\partial M$, there exists a $C^\infty$
simple closed curve $\Gamma^{\infty}$  which is close to $\Gamma$ in $C^0$-topology such that $\Gamma^{\infty}$ bounds a unique area minimizing disk  in $M$.
Similar results holds for the absolutely area minimizing surface case, too. It might be interesting to study these questions in $C^k$-topology.

We should note that the generic uniqueness results in \cite{Wh1} are not directly related with our results. In \cite{Wh1}, for a fixed
$(m-1)$-manifold $X$, White shows that a generic element in $C^{j,\alpha}$ embeddings of $X$ into $\BR^n$ bounds a unique absolutely area
minimizing $m$-manifold in $\BR^n$ (\cite{Wh1}, Theorem 7). In particular, this result implies that a generic $C^{j,\alpha}$ simple closed
curve in $\BR^3$ bounds a unique absolutely area minimizing surface \cite{Mo}. White's result also generalizes to closed manifolds of any dimension
(see Section $8$ in \cite{Wh1}). However, it does not generalize to the manifolds with boundary (see the remarks in Section $8$ in \cite{Wh1}).
Hence, although it implies generic uniqueness for the curves in the interior of the manifold, it does not imply even the existence of a uniqueness curve
in $\partial M$. In this sense, White's results are not directly related with the results in this paper. On the other hand, it might be interesting
to generalize White's techniques to the manifolds with boundary, and hence to solve the generic uniqueness question in the smooth category mentioned
above.

\section*{Acknowledgements}

The authors would like to thank Joel Hass, Antonio Ros, and Theodora Bourni for very useful conversations. The second author gratefully acknowledges
the support of the Max Planck Institute in Bonn where part of this paper was written.

\end{document}